\documentclass[11pt]{amsart}
\usepackage{amsmath}
\usepackage{amssymb}
\usepackage{amsfonts}
\usepackage{mathrsfs}
\usepackage{graphicx}
\usepackage{epsfig}

\theoremstyle{plain}

\newtheorem{prop}{Proposition}

\newtheorem{lemm}[prop]{Lemma}
\newtheorem{fact}{Fact}
\newtheorem{ques}{Question}
\newtheorem{theoalph}{Theorem}

\theoremstyle{definition}

\theoremstyle{remark}
\newtheorem{rema}[prop]{Remark}

\newtheoremstyle{citing}
  {3pt}
  {3pt}
  {\itshape}
  {}
  {\bfseries}
  {.}
  {.5em}
  {\thmnote{#3}}

\theoremstyle{citing}

\newcommand{\partn}[1]{{\smallskip \noindent \textbf{#1.}}}

\DeclareMathAlphabet{\mathpzc}{OT1}{pzc}{m}{it} 

\newcommand{\C}{\mathbb{C}}

\newcommand{\R}{\mathbb{R}}

\newcommand{\cA}{\mathcal{A}}

\newcommand{\cO}{\mathcal{O}}

\newcommand{\sC}{\mathscr{C}}

\newcommand{\hB}{\widehat{B}}

\newcommand{\hU}{\widehat{U}}
\newcommand{\hV}{\widehat{V}}

\newcommand{\hvarphi}{\widehat{\varphi}}

\newcommand{\tB}{\widetilde{B}}

\newcommand{\tV}{\widetilde{V}}

\newcommand{\teta}{\widetilde{\teta}}

\newcommand{\tvarphi}{\widetilde{\varphi}}

\DeclareMathOperator{\mmod}{mod} 

\renewcommand{\=}{ : = }

\DeclareMathOperator{\diam}{diam}

\DeclareMathOperator{\dist}{dist}

\DeclareMathOperator{\Crit}{Crit}

\newcommand{\CC}{\overline{\C}}

\newcommand{\TCEC}{TCE condition}
\newcommand{\ExpShrink}{ESC}
\DeclareMathOperator{\per}{per}

\begin{document}

\title[Maximal entropy measure and non-uniform hyperbolicity]{The maximal entropy measure detects non-uniform hyperbolicity}
\author[J. Rivera-Letelier]{Juan Rivera-Letelier$^\dag$}
\thanks{$\dag$Partially supported by the Research Network on Low Dimensional Dynamics, PBCT ACT-17, CONICYT, Chile, and the \emph{Centre National de la Recherce Scientifique} (CNRS), France, \emph{unit{\'e}} FR2291 FRUMAM} 
\address{$\dag$ Juan Rivera-Letelier, Facultad de Matem{\'a}ticas, Campus San Joaqu{\'\i}n, P.~Universidad Cat{\'o}lica de Chile, Avenida Vicu{\~n}a Mackenna~4860, Santiago, Chile}
\email{riveraletelier@mat.puc.cl}

\begin{abstract}
We characterize two of the most studied non-uniform hyperbolicity conditions for rational maps, semi-hyperbolicity and the topological Collet-Eckmann condition, in terms of the maximal entropy measure.

Using the same tools in the proof of these results we give an extension of a result of Carleson, Jones and Yoccoz, that semi-hyperbolicity characterizes those polynomial maps whose basin of attraction of infinity is a John domain, to rational maps having a completely invariant attracting basin.
\end{abstract}
\subjclass[2000]{Primary:  37F10; Secondary: 30D05, 37D25, 30C20}
\keywords{Non-uniform hyperbolicity, maximal entropy measure, Julia set, doubling measure}

\maketitle

%
%

\section{Introduction}
Two of the most studied non-uniform hyperbolicity conditions for complex rational maps can be formulated in topological terms.
The first is ``semi-hyperbolicity'', that was introduced by Carleson, Jones and Yoccoz to characterize those complex polynomials whose basin of attraction of infinity is a John domain, see~\cite{CarJonYoc94}.
The second is the ``Topological Collet-Eckmann'' condition, that was introduced in the context of rational maps by Przytycki and Rohde in~\cite{PrzRoh98}.
Graczyk and Smirnov~\cite{GraSmi98} and Przytycki~\cite{Prz00} showed that this condition characterizes those polynomials whose basin of attraction of infinity is a H{\"o}lder domain.

In this paper we characterize each of these conditions in terms of the maximal entropy measure.
Recall that each rational map of degree at least two possesses a unique invariant probability measure of maximal entropy and that this measure is supported on the Julia set of the rational map~\cite{FreLopMan83,Man83,Lju83}.
For a polynomial, the maximal entropy measure coincides with the harmonic measure of its Julia set.

To state our main results, let~$f$ be a rational map of degree at least two and fix a small radius~$r > 0$.
For a point~$x$ in the Riemann sphere~$\CC$ and an integer~$m \ge 1$ we define the \emph{semi-local degree of~$f^m$ at~$x$} as follows.
Let~$W$ be the connected component of~$f^{-m}(B(f^m(x), r))$ containing~$x$.
Then~$f^m : W \to B(f^m(x), r)$ is a ramified covering and the semi-local degree of~$f^m$ at~$x$ is by definition the degree of this map.

The rational map~$f$ is said to be \emph{semi-hyperbolic}, if for a sufficiently small~$r > 0$ there is a constant~$D \ge 1$, such that for each integer $m \ge 1$ the semi-local degree of~$f^m$ at each point of~$J(f)$ is less than or equal to~$D$.

Recently, Ha{\"{\i}}ssinsky and Pilgrim showed that the measure of maximal entropy of a semi-hyperbolic rational map is doubling on the Julia set, see~\cite[Proposition~$4.2.9$]{HaiPil0612}.
Recall that a Borel measure~$\rho$ on a metric space $(X, \dist)$ is said to be doubling, if there are constants~$C_* > 0$ and $r_* > 0$ such that for each~$x \in X$ and~$r \in (0, r_*)$ we have
$$ \rho(B(x, 2r)) \le C_* \rho(B(x, r)). $$
Our first result is that in fact this property of the maximal entropy measure characterizes semi-hyperbolicity.
\begin{theoalph}\label{t:semi-hyperbolicity}
A complex rational map of degree at least two is semi-hyperbolic if and only if its maximal entropy measure is doubling on the Julia set.
\end{theoalph}
Combining this result with the main result of~\cite{CarJonYoc94}, we obtain that if a John domain is the basin of attraction of infinity of a complex polynomial, then its harmonic measure is doubling on the boundary.
This result should be compared with a result of Kim and Langmeyer~\cite[Theorem~$2.3$]{KimLan98}, that a bounded Jordan domain is a John domain if and only if its harmonic measure is doubling on the boundary.
Note however that there are planar simply-connected John domains for which the harmonic measure is not doubling on the boundary~\cite{BalVol96a}.

A rational map~$f$ satisfies the \emph{Topological Collet-Eckmann (TCE) condition}, if for some~$r > 0$ there are constants $D \ge 1$ and~$\theta \in (0, 1)$ such that the following property holds.
For each~$x \in J(f)$ the set~$G_x$ of those integers~$m \ge 1$ for which the semi-local degree of~$f^m$ at~$x$ is less than or equal to~$D$ satisfies,
$$ \liminf_{n \to + \infty} \frac{1}{n} \# (G_x \cap \{ 1, \ldots, n \}) \ge \theta. $$
Clearly, every semi-hyperbolic rational map satisfies the TCE condition.
\begin{theoalph}\label{t:TCE}
Let~$f$ be a complex rational map of degree at least two and let~$\rho_f$ be the maximal entropy measure of~$f$.
Then~$f$ satisfies the \TCEC{} if and only if there are constants~$r_0 > 0$, $\alpha > 0$ and~$C > 0$ such that for all~$x \in J(f)$ and~$r \in (0, r_0)$ we have
$$ \rho_f(B(x, r)) \ge C r^{\alpha}. $$
\end{theoalph}
This result adds yet another characterization of the TCE condition to those given in~\cite{PrzRivSmi03}.
See also~\cite[Corollary~$1.1$]{PrzRiv07}.

We determine the optimal constant~$\alpha$ appearing in 
the statement of this theorem, see Remark~\ref{r:optimal constant}.

\subsection{Semi-hyperbolicity and John domains}
\label{ss:furtherresults}
Our final result is an extension to rational maps of the result of Carleson, Jones and Yoccoz for polynomials.
It is not directly related to the previous results, but we have included it here as its proof uses some of the tools developed to prove Theorem~\ref{t:semi-hyperbolicity}.
Note that for a polynomial the basin of attraction of infinity is completely invariant.
\begin{theoalph}\label{t:John domains}
Let~$f$ be a rational map of degree at least two having a completely invariant attracting basin~$\cA$.
Then~$\cA$ is a John domain if and only if~$f$ is semi-hyperbolic.
\end{theoalph}
There are simple examples showing that the hypothesis that the completely invariant Fatou component~$\cA$ is an attracting basin is necessary, see~\S\ref{ss:notes}.

One of the implications of this theorem is given by an extension of one of the results of Carleson, Jones and Yoccoz, shown by Mihalache in~\cite{Mih0803}: each Fatou component of a semi-hyperbolic rational map is a John domain, see also~\cite{Yin99} for the case of connected Julia sets.
See~\S\ref{ss:notes} for several examples of rational maps showing that the converse of this last result does not hold in general.

To prove the reverse implication we use the (straightforward) fact that every John domain is porous.
Recall that a subset~$J$ of~$\C$ is \emph{porous} if there is~$\xi \in (0, 1)$ such that the following property holds: for each sufficiently small~$r > 0$ and each~$x \in J$ there is $y \in B(x, r)$ such that the ball~$B(y, \xi r)$ is disjoint from~$J$.
The key step in the proof is to show that if~$\cA$ is porous, then there are no recurrent critical points in the Julia set (Lemma~\ref{l:non-porosity}).
Then we conclude using the extension to rational maps of one of the results of Carleson, Jones and Yoccoz~\cite{CarJonYoc94}, given by Yin in~\cite{Yin99}: a rational map is semi-hyperbolic if and only if it has neither parabolic periodic points nor recurrent critical points in the Julia set.
An important preliminary step in the proof of Theorem~\ref{t:John domains} is to show that the map satisfies the TCE condition.
This follows from the fact that each John domain is a H{\"o}lder domain and from results in~\cite{GraSmi98,PrzRivSmi03}.
In fact we show the following stronger version of Theorem~\ref{t:John domains}: $\cA$ is a H{\"o}lder domain with a porous boundary if and only if~$f$ is semi-hyperbolic.

We end the introduction with a question, formulated with the intent of understanding further the connection between the geometry of Julia sets and the non-uniform expansion of the corresponding maps.
As remarked above the Julia set of a semi-hyperbolic polynomial is porous.
However, there are polynomials having a porous Julia set that are not semi-hyperbolic.
See~\cite{PrzUrb01,Yin00} for an example with a parabolic periodic point and~\cite[Theorem~$4.1$]{McM98} for one with a Siegel disk.

It would be interesting to know if the following variant of porosity characterizes semi-hyperbolicity.
We say that a subset~$K$ of~$\C$ is \emph{boundary porous} if there is a constant~$\xi \in (0, 1)$ such that for each sufficiently small $r > 0$ and each~$x \in \partial K$ there is~$y \in B(x, r)$ such that the ball~$B(y, \xi r)$ is disjoint from~$K$.
\begin{ques}\label{q:porosity}
Let~$f$ be a polynomial whose filled-in Julia set is boundary porous.
Is~$f$ semi-hyperbolic?
\end{ques}

\subsection{Notes and references}
\label{ss:notes}
Each doubling measure satisfies the property described in Theorem~\ref{t:TCE}, see Lemma~\ref{l:doubling is regular}.
An analogous upper bound holds for the maximal entropy measure of each rational map~$f$: there are $C' > 0$, $\alpha' > 0$ such that for all $x \in J(f)$ and $r > 0$ we have
$$ \rho_f(B(x, r)) \le C' r^{\alpha'}, $$
see for example~\cite[Lemma~$4$]{PrzUrbZdu91}.

See~\cite[Theorem~$4.2.3$, Theorem~$4.2.8$]{HaiPil0612} and~\cite[Theorem~$8.1$]{LyuMin97} for other characterizations of semi-hyperbolic rational maps and~\cite{BalVol96a} for a refinement of the result of Carleson, Jones and Yoccoz.
Mihalache showed in~\cite{Mih0803} that each Fatou component of a semi-hyperbolic rational map is a John domain with a uniform constant, see also~\cite[Theorem~$1.2$]{Yin99} for the case where the Julia set is connected.
 
There are several examples showing the converse of this last result does not hold in general.
Perhaps the simplest is the rational map~$R(z) = \frac{1}{z} - z$.
It is not semi-hyperbolic because~$z = \infty$ is a parabolic fixed point of~$R$.
On the other hand the Fatou set of~$R$ consists of the upper and lower half-plane, both of which are John domains in~$\CC$.
Since each of these components is completely invariant by~$R$, this example also shows that in Theorem~\ref{t:John domains} the hypothesis that the completely invariant Fatou component~$\cA$ is an attracting basin is necessary.
There are similar examples of any given degree, see for example~\cite{Yin99}.

Another interesting example, pointed out in~\cite{Mih0803}, is given by the mating of quadratic polynomials with a Siegel disk, that was studied by Yampolsky and Zakeri in~\cite{YamZak01}.
In fact, this rational map is not semi-hyperbolic as it has a Siegel disk and yet each of its Fatou components is a quasi-disk and hence a John domain.

To give a different class of examples we consider the following direct consequence of the results of Roesch in~\cite{Roe08}, see~\S\ref{s:John domains} for the proof.
\begin{fact}
\label{f:cubic Newton}
Let~$N$ be a twice renormalizable Newton method of a cubic polynomial with simple roots.
Suppose furthermore that~$N$ has no parabolic periodic points or Siegel disks.
Then the Fatou components of~$N$ are quasi-disks with a uniform constant.
\end{fact}
A direct consequence of this fact is that the Fatou components of~$N$ are John domains with a uniform constant.
Notice that the rational map~$N$ can be chosen so as to have a Cremer periodic point or to have no neutral cycles and a recurrent critical point in the Julia set.

\subsection{Acknowledgements}
\label{s:acknowledgments}
The main ideas of this paper came to the author after several discussions with Peter Ha{\"{\i}}ssinsky and Kevin Pilgrim on their recent monograph~\cite{HaiPil0612}.
I thank both of them for those stimulating conversations and for their useful comments and corrections to an earlier version of this paper.
Nicolae Mihalache read an earlier version very carefully and made several corrections.
Several discussions with him and with Feliks Przytycki and Mariusz Urbanski were useful.
My gratitude goes to all of them.

Finally, I would like to thank the \emph{Centre de Math{\'e}matiques et d'Informatique} of \emph{Universt{\'e} de Provence} for hospitality while part of this research was done.

\section{Preliminaries}
\label{s:preliminaries}

We endow~$\CC$ with the spherical metric, that we denote by~$\dist$.
Unless otherwise stated, distances, balls, diameters and derivatives, will be all taken with respect to the spherical metric.

Given a rational map~$f$, an integer~$m \ge 1$ and a subset~$V$ of~$\CC$, a connected component of~$f^{-m}(V)$ will be called a \emph{pull-back of~$V$ by~$f^m$}.
\subsection{Critical points}
\label{ss:critical points}
Given a complex rational map~$f$ we denote by~$\Crit(f)$ the set of critical points of~$f$ and by~$\Crit'(f)$ the set of those critical points of~$f$ which are in the Julia set.
We will say that~$f$ has \emph{critical connections} if there is an integer~$n \ge 1$ such that~$f^n(\Crit'(f)) \cap \Crit'(f) \neq \emptyset$.
We will say that a critical point~$c \in \Crit'(f)$ is \emph{exposed}, if for each integer~$n \ge 1$ we have~$f^n(c) \not \in \Crit(f)$.

We denote by~$\ell_{\max}(f)$ the maximal local degree of~$f$ at a critical point in~$J(f)$ and put
$$ \widehat{\ell}_{\max}(f) \= \max \{ \ell_{\max}(f^n) : n \ge 1 \}. $$
Note that~$\widehat{\ell_{\max}}(f) \le 2^{2 \deg(f) - 2}$ and that~$\widehat{\ell_{\max}}(f) = \ell_{\max}(f)$ in the case where~$f$ does not have critical connections.

\subsection{Maximal entropy measure}
\label{ss:maximal entropy measure}
As noted before, each rational map~$f$ of degree at least two has a unique measure of maximal entropy.
We will denote this measure by~$\rho_f$.
The topological support of~$\rho_f$ is equal to~$J(f)$ and the Jacobian of~$\rho_f$ is constant equal to~$\deg(f)$.
We will use several times the following property of~$\rho_f$.
\begin{fact}\label{f:Jacobian}
Let~$f$ be a rational map and let~$\rho_f$ be its maximal entropy measure.
Let~$V$ be an open and connected subset of~$\CC$, let $m \ge 1$ be an integer and let~$W$ be a pull-back of~$V$ by~$f^m$.
If we denote by~$D$ the degree of $f^m : W \to V$, then
$$ \rho_f(W) = D \deg(f)^{-m} \rho_f(V). $$
\end{fact}
This property is a direct consequence of the fact that the Jacobian of~$\rho_f$ is constant equal to~$\deg(f)$ and of the fact that~$\rho_f$ does not charge points.

\subsection{Doubling measures}
\label{ss:doubling measures}
We will use the following property of doubling measures.
\begin{lemm}\label{l:doubling is regular}
Let~$(X, d)$ be a compact metric space and let~$\rho$ be a doubling measure on~$X$.
Then there are constants~$C > 0$ and~$\alpha > 0$ such that for each sufficiently small~$r > 0$ and each~$x \in X$ we have
$$ \rho(B(x, r)) \ge C r^{\alpha}. $$
\end{lemm}
\begin{proof}
Let~$r_* > 0$ and~$C_* > 0$ be constants associated to the doubling property of~$\rho$ and let~$\varepsilon > 0$ be such that for each~$x \in X$ we have~$\rho(B(x, r_* / 2)) \ge \varepsilon$.
Given~$r \in (0, r_*)$, let~$n \ge 0$ be the unique integer such that~$2^n r < r_* \le 2^{n + 1}r$.
Then
$$ \rho(B(x, r))
\ge
C_*^{-n} \rho(B(x, 2^n r))
\ge
C_*^{-n} \varepsilon. $$
This shows the desired assertion with~$\alpha = \ln C_* / \ln 2$ and~$C = r_*^{-\alpha} \varepsilon$.
\end{proof}
A compact subset~$J$ of the Riemann sphere is \emph{uniformly perfect}, if there are~$\hat{\eta} > 1$ and~$\hat{r} > 0$ such that for each~$x \in J(f)$ and each~$r \in (0, \hat{r})$ the annulus~$B(x, \hat{\eta} r) \setminus B(x, r)$ intersects~$J$.
\begin{lemm}\label{l:inverse doubling}
Let~$J$ be a uniformly perfect compact subset of~$\CC$ and~$\rho$ a doubling measure supported on~$J$.
Then the following properties hold.
\begin{itemize}
\item[1.]
There are $\eta_0 > 1$, $\varepsilon_0 \in (0, 1)$ and~$r_0 > 0$ such that for each~$r \in (0, r_0)$ and~$x \in J$ we have
$$ \rho(B(x, \eta_0 r) \setminus B(x, r))
\ge
\varepsilon_0 \rho(B(x, r)). $$
\item[2.]
There are $\eta_1 > 1$ and~$r_1 > 0$ such that for each~$r \in (0, r_1)$ and~$x \in J$ we have
$$ \rho (B(x, \eta_1r))
\ge
2 \rho(B(x, r)). $$
\end{itemize}
\end{lemm}
\begin{proof}
\

\partn{1}
Let~$C_* > 0$ and~$r_* > 0$ be the constants associated to the doubling property of~$\rho$ and let~$\hat{\eta} > 1$ and $\hat{r} > 0$ be the constants associated to the uniform perfectness of~$J$.
Let~$n \ge 1$ be a sufficiently large integer such that~$2^n \ge 2(\hat{\eta} + 1)$.

Given~$x \in J$ and~$r \in (0, \min \{2^{-n} r_*, 2^{-1} \hat{r} \})$, let~$x' \in B(x, 2 \hat{\eta} r) \setminus B(x, 2r)$ be in~$J$.
Since~$\rho$ is doubling on~$J$ we have
$$ \rho(B(x', 2(\hat{\eta} + 1)r))
\le
\rho(B(x', 2^nr))
\le
C_*^n \rho(B(x', r)). $$
As~$B(x, r) \subset B(x', 2(\hat{\eta} + 1) r)$, we conclude that~$\rho(B(x', r)) \ge C_*^{-n} \rho(B(x, r))$.
On the other hand, using that~$x' \in B(x, 2\hat{\eta} r) \setminus B(x, 2r)$ we have
$$ B(x', r) \subset B(x, (2\hat{\eta} + 1) r) \setminus B(x, r) $$
and hence
$$ \rho(B(x, (2\hat{\eta} + 1) r) \setminus B(x, r))
\ge
C_*^{-n} \rho(B(x, r)). $$
This shows that the desired property holds with~$\eta_0 = 2\hat{\eta} + 1$, $\varepsilon_0 = C_*^{-n}$ and~$r_0 = \min \{2^{-n} r_*, 2^{-1} \hat{r} \}$.

\partn{2}
Let~$\eta_0 > 1$, $\varepsilon_0 \in (0, 1)$ and $r_0 > 0$ be given by part~1 and let~$n \ge 1$ be a sufficiently large integer such that~$(1 + \varepsilon_0)^n \ge 2$.
Using part~1 inductively, we obtain that for each~$r \in (0, \eta_0^{-n}r_0)$,
$$ \rho(B(x, \eta_0^n r))
\ge
(1 + \varepsilon_0)^n \rho(B(x, r))
\ge
2 \rho (B(x, r)). $$
This shows the desired property with~$\eta_1 = \eta_0^n$ and $r_1 = \eta_0^{-n} r_0$.
\end{proof}
\subsection{Distortion lemma}
\label{ss:distortion lemmas}
The following geometric lemma is a direct consequence of Koebe distortion theorem.
We omit the proof.
\begin{lemm}\label{l:middle is round}
Given~$R > r> 0$ put
$$ A(r, R) \= \{ z \in \C : r < |z| < R \}. $$
Then there are constants~$M > 1$ and~$\delta > 0$ such that for each univalent map
$$ \varphi : A(1, 8) \to \CC, $$
satisfying~$\diam(\varphi(A(1, 8))) < \delta$ and for each~$x$ enclosed by the image of~$\varphi$, we have
$$ \frac{\sup \{ \dist(x, y) : y \in \varphi(A(2, 4)) \} }{ \inf \{ \dist(x, y) : y \in \varphi(A(2, 4)) \}}
\le M. $$
\end{lemm}

\section{TCE condition}
\label{s:TCE}
This section is devoted to the proof of Theorem~\ref{t:TCE}.
We will use the fact that the TCE condition is characterized by each of the following conditions, see~\cite{PrzRivSmi03}.
Let~$f$ be a rational map of degree at least two.

\begin{quote}
\textbf{Exponential shrinking of components (\ExpShrink{}).}
There are~$r_0 > 0$ and $\lambda > 1$ such that for each~$x \in J(f)$ and each integer~$m \ge 1$, each connected component~$W$ of~$f^{-m}(B(x, r))$ satisfies
$$ \diam(W) \le \lambda^{-m}. $$
\end{quote}
Recall that given an integer~$n \ge 1$, a periodic point~$p$ of period~$n$ of~$f$ is \emph{repelling} if~$|(f^n)'(p)| > 1$.
\begin{quote}
\textbf{Uniform hyperbolicity on periodic orbits.}
There is~$\lambda > 1$ such that for each integer~$n \ge 1$ and each repelling periodic point~$p$ of period~$n$ we have~$|(f^n)'(p)| \ge \lambda^n$.
\end{quote}
\begin{proof}[Proof of Theorem~\ref{t:TCE}]
Let~$f$ be a rational map satisfying the TCE condition and let~$r_0 > 0$ and $\lambda > 1$ be the constants given by the \ExpShrink{} condition.
Let~$\varepsilon > 0$ be sufficiently small so that for every~$x \in J(f)$ we have~$\rho_f(B(x, r_0)) > \varepsilon$.

Let~$x \in J(f)$ and $\widehat{r} \in (0, 1)$ be given and let~$n \ge 1$ be the integer such that~$\lambda^{-n} \le \widehat{r} < \lambda^{-(n - 1)}$.
Let~$W$ be the connected component of~$f^{-n}(B(f^n(x), r_0))$ containing~$x$ and let~$D \ge 1$ be the degree of~$f^n : W \to B(f^n(x), r_0)$.
Then, by Fact~\ref{f:Jacobian} we have
$$
\rho_f(W)
=
D \deg(f)^{-n} \rho_f(B(f^n(x), r_0))
\ge
\varepsilon \deg(f)^{-n}.
$$
By the \ExpShrink{} condition it follows that~$\diam(W) \le \lambda^{-n}$, so~$W \subset B(x, \widehat{r})$.
Thus, if we put~$\alpha \= \frac{ \ln \deg(f) }{ \ln \lambda }$ and $C \= \varepsilon \lambda^{- \alpha}$, then 
$$ \rho_f(B(x, \widehat{r}))
\ge
\varepsilon \deg(f)^{- n}
=
\varepsilon (\lambda^{- n})^\alpha
\ge 
C \widehat{r}^\alpha. $$
This shows the desired property of~$\rho_f$.

Suppose now that~$f$ is a rational map for which there are constants~$r_0 > 0$, $\alpha > 0$ and $C > 0$, such that for each $x \in J(f)$ and each~$r \in (0, r_0)$ we have
$$ \rho_f(B(x, r)) \ge C r^{\alpha}. $$
We will show that~$f$ is uniformly hyperbolic on periodic orbits.
As remarked above, this implies that~$f$ satisfies the TCE condition.
Let~$n \ge 1$ be an integer and let~$p$ be a repelling periodic point of~$f$ of period~$n$.
Then there is a local inverse~$\varphi$ of~$f^n$ which is defined on a neighborhood of~$p$ and which fixes~$p$.
Furthermore, if~$r_1 > 0$ is sufficiently small, then~$\varphi$ is defined on~$B(x, r_1)$ and there is a constant~$C_1 > 0$ such that for each~$k \ge 1$ we have
$$ B(p, C_1 |(f^n)'(p)|^{-k}) \subset \varphi^{k}(B(p, r_1)). $$
By Fact~\ref{f:Jacobian} we have
\begin{multline*}
\deg(f)^{-kn} \rho_f(B(p, r_1))
=
\rho_f(\varphi^{k}(B(p, r_1)))
\\ \ge
\rho_f(B(p, C_1 |(f^n)'(p)|^{-k}))
\ge
CC_1^\alpha |(f^n)'(p)|^{-k \alpha}.
\end{multline*}
Since this holds for every integer~$k \ge 1$, it follows that~$|(f^n)'(p)| \ge \deg(f)^{n / \alpha}$.
This shows that~$f$ is uniformly hyperbolic on periodic orbits with constant~$\lambda = \deg(f)^{1 / \alpha}$.
\end{proof}
\begin{rema}\label{r:optimal constant}
For a rational map~$f$ satisfying the TCE condition we will now determine the optimal constant~$\alpha$ in Theorem~\ref{t:TCE}.
For each integer $n \ge 1$ and each periodic point~$p$ of period~$n$, put
$$ \chi(p) = \frac{1}{n} \ln |(f^n)'(p)|, $$
and
$$
\chi_{\per}
\=
\inf \{ \chi(p) : p \text{ repelling periodic point of~$f$} \}.
$$
Then, in~\cite{PrzRivSmi03} it is shown that \ExpShrink{} holds for each~$\lambda \in (1, \exp(\chi_{\per}))$.

Thus, the proof of Theorem~\ref{t:TCE} gives that if~$\alpha$ is a constant for which the conclusion of this theorem holds, then~$\alpha \ge \ln \deg(f) / \chi_{\per}$.
On the other hand, we proved that the conclusion of Theorem~\ref{t:TCE} holds for each~$\alpha > \ln \deg(f) / \chi_{\per}$.
\end{rema}
\section{Semi-hyperbolicity}
\label{s:semi-hyperbolicity}
This section is devoted to the proof of Theorem~\ref{t:semi-hyperbolicity}.
The proof is based on Lemma~\ref{l:thin unramified annulus} below and makes use of the fact that a rational map satisfying the TCE condition has arbitrarily small ``nice couples'', as shown in~\cite{PrzRiv07}.
We will also use the fact that a rational map is semi-hyperbolic if and only if it has neither parabolic periodic points nor recurrent critical points in the Julia set.
This was shown by Carleson, Jones and Yoccoz in~\cite{CarJonYoc94} for polynomials and then it was extended to rational maps by~Yin in~\cite[Theorem~$1.1$]{Yin99}.

Throughout all this section we fix a complex rational map~$f$ of degree at least two.

An open neighborhood~$V$ of~$\Crit'(f)$ that is disjoint from the forward orbit of critical points not in~$J(f)$ is called a \emph{nice set for~$f$}, if for each integer~$n \ge 1$ we have~$f^n(\partial V) \cap V = \emptyset$ and if each connected component of~$V$ is simply-connected and contains precisely one element of~$\Crit'(f)$.
We say that a pair of nice sets $(\hV, V)$ is a \emph{nice couple for~$f$}, if~$\overline{V} \subset \hV$ and if for each integer~$n \ge 1$ we have~$f^n(\partial V) \cap \hV = \emptyset$.
We will say that a nice couple~$(\hV, V)$ for~$f$ is \emph{small}, if there is a small $r > 0$ such that $\hV \subset B(\Crit'(f), r)$.
Given topological disks~$U, \hU \subset \CC$ such that~$\overline{U} \subset \hU$, we define $\mmod(\hU;U)$ as the supremum of the modulus of those annuli that separate~$U$ and $\CC \setminus \hU$.
If~$f$ is a rational map and if~$(\hV, V)$ is a nice couple for~$f$, then we define the \emph{modulus of $(\hV, V)$} as
$$ \mmod(\hV; V) \= \min \{ \mmod (\hV^c; V^c) : c \in \Crit'(f) \}. $$

In the proof of Theorem~\ref{t:semi-hyperbolicity} we will use the fact that a rational map satisfying the TCE condition has arbitrarily small nice couples of arbitrarily large modulus, see~\cite[Proposition~$4.2$]{PrzRiv07}.
\begin{lemm}\label{l:thin unramified annulus}
Let~$f$ be a rational map of degree at least two having arbitrarily small nice couples of arbitrarily large modulus.
Then for each recurrent critical point~$c_0$ in~$J(f)$, $ \kappa \in (0, 1)$, $N \ge 2$ and each $r_* > 0$ there is~$c \in \Crit'(f)$, $r \in (0, r_*)$ and an integer~$m \ge 1$, such that~$f^m(c_0) \in B(c, r)$ and such that the pull-back~$\hU$ (resp.~$U$) of~$B(c, r)$ (resp.~$B(c, \kappa r)$) containing~$c_0$ satisfies the following properties.
\begin{enumerate}
\item[1.]
$\diam(\hU) < r_*$.
\item[2.]
The degree of~$f^m$ on~$\hU$ and the degree of~$f^m$ on~$U$ are the same.
\item[3.]
The set $A \= \hU \setminus \overline{U}$ is an annulus and the map
$$ f^m : A \to B(c, r) \setminus \overline{B(c, \kappa r)} $$
is a covering map whose degree is at least~$N$ and at most~$\widehat{\ell_{\max}}(f) N$.
\end{enumerate}
\end{lemm}
To prove this lemma we will make the following definition.
Let~$f$ be a rational map of degree at least two and let~$\sC$ be a subset of~$\Crit'(f)$.
We will say that an open neighborhood~$V$ of~$\sC$ which is disjoint from the forward orbits of the critical points which are not in~$J(f)$ is a \emph{nice set for~$f$ relative to~$\sC$}, if for each integer~$n \ge 1$ we have~$f^n(\partial V) \cap V = \emptyset$, and if each connected component of~$V$ is simply-connected and contains precisely one element of~$\sC$.
Notice that a nice set for~$f$ is a nice set for~$f$ relative to~$\Crit'(f)$.
Given nice sets~$\hV$ and~$V$ for~$f$ relative to~$\sC$ satisfying~$\overline{V} \subset \hV$, we will say that~$(\hV, V)$ is a \emph{nice couple for~$f$ relative to~$\sC$}, if for each integer~$n \ge 1$ we have~$f^n(\partial V) \cap \hV = \emptyset$.
\begin{proof}[Proof of Lemma~\ref{l:thin unramified annulus}]
If~$f$ has critical connections, then by replacing~$c_0$ by a critical point in its forward orbit if necessary, we assume that~$c_0$ is exposed.
Let~$\sC_0$ be the set of all those critical points of $f$ in~$J(f)$ whose forward orbit accumulates on~$c_0$.
Note that for each~$c \in \Crit(f) \setminus \sC_0$ the $\omega$-limit set of $c$ is disjoint from~$\sC_0$.
Thus there is $r_0 > 0$ such that if for some $c \in \Crit(f)$ and some integer $n \ge 1$ we have $f^n(c) \in B(\sC_0, r_0)$, then~$c \in \sC_0$.

Fix a periodic orbit~$\cO$ of period at least two of~$f$, disjoint from~$\Crit'(f)$.
Let~$\mathsf{m}_0 > 0$ be sufficiently large so that the following property holds.
If~$W \subset \CC$ is a topological disk disjoint from~$\cO$ and if~$K \subset W$ is a compact set such that~$W \setminus K$ is an annulus of modulus at least~$\mathsf{m}_0$, then~$\diam(K) < r_*$ and for each~$x \in K$ there is~$r \in (0, r_*)$ such that~$K \subset B(x, \kappa r)$ and $B(x, r) \subset W$, see for example~\cite[Proposition~$2.1$]{McM94}.

Let~$n \ge 1$ be sufficiently large integer such that~$2^n \ge N$ and fix a sufficiently small nice couple~$(\hV, V)$ such that~$\hV$ is disjoint from~$\cO$, contained in~$B(\Crit'(f), r_0)$ and such that~$\mmod(\hV; V) \ge 3\widehat{\ell_{\max}}(f)^n \mathsf{m}_0$.
Note that each pull-back~$W$ of~$\hV$ is disjoint form~$\cO$.
Furthermore, if for some integer~$m \ge 1$ and~$c \in \Crit(f)$ we have $f^m(c) \in B(\sC_0, r_0)$, then~$c \in \sC_0$.

\partn{1}
Let~$\sC_0'$ be the set of those exposed critical points in~$\sC_0$.
We have $c_0 \in \sC_0'$.
We will construct by induction a sequence of nice couples~$( (\hV_n, V_n) )_{n \ge 0}$ for~$f$ relative to~$\sC_0'$, as follows.
For each $c \in \sC_0'$ denote by~$\hV_0^c$ (resp.~$V_0^c$) the connected component of~$\hV$ (resp.~$V$) containing~$c$ and put
$$ \hV_0 \= \bigcup_{c \in \sC_0'} \hV_0^c
\text{ and }
V_0 \= \bigcup_{c \in \sC_0'} V_0^c. $$
Clearly~$(\hV_0, V_0)$ is a nice couple for~$f$ relative to~$\sC_0'$.

Let~$n \ge 0$ be a given integer and suppose by induction that $(\hV_n, V_n)$ is already defined.
As the forward orbit of each $c \in \sC_0'$ accumulates on~$c_0$ there is an integer $m \ge 1$ such that $f^m(c) \in V_n$.
Let $m_n(c)$ be the least such integer and let~$\hV_{n + 1}^c$ (resp. $V_{n + 1}^c$) be the pull-back of~$\hV_n$ (resp.~$V_n$) by~$f^{m_n(c)}$ containing~$c$.
Clearly $\overline{V_{n + 1}^c} \subset \hV_{n + 1}^c$ and since $(\hV_n, V_n)$ is a nice couple for~$f$ relative to~$\sC_0'$, it follows that~$\hV_{n + 1}^c \subset V_n$ and that
$$ \hV_{n + 1} \= \bigcup_{c \in \sC_0'} \hV_{n + 1}^c
\text{ and }
V_{n + 1} \= \bigcup_{c \in \sC_0'} V_{n + 1}^c, $$
form a nice couple for~$f$ relative to~$\sC_0'$.

\partn{2}
Let~$n \ge 1$ be given.
Using the fact that $(\hV_n, V_n)$ is a nice couple for~$f$ relative to~$\sC_0'$, it follows that if~$f$ does not have critical connections, then~$f^{m_n(c) - 1}$ is univalent on~$f(\hV_n^c)$.
If~$f$ does have critical connections, then~$f^{m_n(c) - 1}$ might not be univalent on~$f(\hV_n^c)$, but in this case~$f^{m_n(c) - 1}$ is unicritical on this set and its unique critical value is contained in~$\sC_0'$.
In all cases it follows that~$f^{m_n(c)}$ does not have critical points in~$\hV_{n + 1}^c \setminus V_{n + 1}^c$ and the degree of~$f^{m_n(c)}$ on~$\hV_n$ and the degree of~$f^{m_n(c)}$ on~$V_n$ are the same.

\partn{3}
Put~$m_0 = 0$ and for each~$j \in \{1, \ldots, n \}$ let~$c_j \in \sC_0'$ and~$m_j \ge 1$ be such that~$f^{m_j}(\hV_n^{c_0}) = \hV_{n - j}^{c_j}$.
Furthermore, put $d_0 = 1$ and for each~$j \in \{0, \ldots, n - 1 \}$ denote by~$d_j$ the degree of~$f^{m_n - m_j} : \hV_{n - j}^{c_j} \to \hV_0^{c_n}$.
Then the degree of~$f^{m_n - m_j} : V_{n - j}^{c_j} \to V_0^{c_n}$ is equal to~$d_j$ and for each~$j \in \{0, \ldots, n - 1 \}$ we have
$$ 2d_j \le d_{j + 1} \le \widehat{\ell_{\max}}(f) d_j. $$
Thus~$2^n \le d_n \le \widehat{\ell_{\max}}(f)^n$ and hence there is~$j \in \{ 0, \ldots, n \}$ such that
$$ N \le d_n/d_j \le \widehat{\ell_{\max}}(f) N. $$
Since~$f^{m_n - m_j} : \hV_{n - j}^{c_j} \to \hV_0^{c_n}$ is of degree~$d_j \le \widehat{\ell_{\max}}(f)^j$ and does not have critical points in~$\hV_{n - j}^{c_j} \setminus \overline{V_{n - j}^{c_j}}$, it follows that
$$ \mmod(\hV_{n - j}^{c_j}; V_{n-j}^{c_j})
=
d_j^{-1} \mmod(\hV_0^{c_n}; V_0^{c_n})
\ge
3\mathsf{m}_0 \widehat{\ell_{\max}}(f)^{n - j}. $$
Let~$\tV$ be a topological disk compactly contained in~$\hV_{n - j}^{c_j}$, such that
$$ \overline{V_{n -  j}^{c_j}} \subset \tV,
\mmod(\hV_{n - j}^{c_j}; \tV) \ge \mathsf{m}_0 \widehat{\ell_{\max}}(f)^{n - j}
\text{ and }
\mmod(\tV; V_{n - j}^{c_j}) \ge \mathsf{m}_0. $$
By our choice of~$\mathsf{m}_0$ there is~$r \in (0, r_*)$ such that~$V_{n - j}^{c_j} \subset B(c_j, \kappa r)$ and $B(c_j, r) \subset \tV \subset \hV_{n - j}^{c_j}$.
Let~$\tB$ (resp. $\hU$, $U$) be the pull-back of~$\tV$ (resp. $B(c_j, r)$, $B(c_j, \kappa r)$) by~$f^{m_j}$ containing~$c_0$.
Since the degree of~$f^{m_j} : \hV_n^{c_0} \to \hV_{n - j}^{c_j}$ is less than or equal to~$\widehat{\ell_{\max}}(f)^{n - j}$, it follows that~$\mmod(\hV_n^{c_0}; \tB) \ge \mathsf{m}_0$.
By our choice of~$\mathsf{m}_0$ this implies~$\diam(\hU) \le \diam(\tB) < r_*$.
Since~$V_n^{c_0} \subset U \subset \hU \subset \hV_n^{c_0}$ and since the degree of~$f^{m_j} : \hV_n^{c_0} \to \hV_{n - j}^{c_j}$ is equal to~$d_n/d_j$ and this map does not have critical points in~$\hV_n^{c_0} \setminus \overline{V_n^{c_0}}$, the conclusion of the lemma holds for this choice of~$r$ and for~$c = c_j$ and~$m = m_j$.
\end{proof}

\begin{proof}[Proof of~Theorem~\ref{t:semi-hyperbolicity}]
That the maximal entropy measure of a semi-hyperbolic rational map is doubling on the Julia set was shown by Ha{\"{\i}}ssinsky and Pilgrim in~\cite[Proposition~$4.2.9$]{HaiPil0612}.

To prove the converse statement, let~$f$ be a complex rational map of degree at least two, whose maximal entropy measure~$\rho_f$ is doubling on~$J(f)$, with constants $r_* > 0$ and $C_* > 0$.
By~\cite[Theorem~$1.1$]{Yin99}, to prove that~$f$ is semi-hyperbolic we just need to show that~$f$ has neither parabolic periodic points nor recurrent critical points in the Julia set.
In view of Lemma~\ref{l:doubling is regular} and Theorem~\ref{t:TCE}, $f$ satisfies the TCE condition.
In particular, $f$~does not have parabolic periodic points.
So we just need to show that~$f$ does not have recurrent critical points in the Julia set.

Suppose by contradiction that~$f$ has a recurrent critical point~$c_0$ in the Julia set.
Since the Julia set~$J(f)$ is uniformly perfect~\cite{Hin93,MandRo92}, it follows that the measure~$\rho_f$ satisfies the hypothesis of Lemma~\ref{l:inverse doubling}.
Let~$\eta_1 > 1$ and~$r_1 > 0$ be given by part~2 of Lemma~\ref{l:inverse doubling}.
Let~$M > 0$ and~$\delta > 0$ be the constants given by Lemma~\ref{l:middle is round}.

Since~$\CC$ is endowed with the spherical metric, for small $r' > r > 0$ and $x \in \CC$ the modulus of the annulus $B(x, r') \setminus \overline{B(x, r)}$ is equal to $\ln (r'/r)$ plus an error term that goes to zero as~$r' \to 0$.
So, reducing~$r_* > 0$ if necessary, we assume that~$r_* < \delta$ and that for each~$r \in (0, r_*]$, $\varepsilon \in (0, 1)$ and~$x \in \CC$ we have
\begin{equation}\label{e:modulus defect}
\left| \mmod \left( B(x, r) \setminus \overline{B(x, \varepsilon r)} \right) + \ln \varepsilon \right|
\le
\frac{1}{10}.
\end{equation}

Let~$k, \ell, N \ge 0$ be integers such that~$2^{k} \ge M$, $2^\ell > C_*^{k}$ and
$$ 2^{N/2} \le \eta_1^{\ell} < 2^{(N + 1) / 2}. $$
Taking~$\eta_1$ larger if necessary we assume~$N \ge 2$.

Since~$f$ satisfies the TCE condition, by~\cite[Proposition~$4.2$]{PrzRiv07}~$f$ has arbitrarily small nice couples of arbitrarily large modulus.
So~$f$ satisfies the hypothesis of Lemma~\ref{l:thin unramified annulus}.
Let~$\hU, U, A, m, r, c$ be given by Lemma~\ref{l:thin unramified annulus} for~$N$ and~$c_0$ as above and with
$$ r_* = \delta
\text{ and }
\kappa = \eta_1^{- (1 + 10 \widehat{\ell_{\max}(f)})\ell}. $$
Put $r_0 \= \kappa r$.
By definition~$A = \hU \setminus \overline{U}$ and~$\hU$ (resp.~$U$) is the connected component of~$f^{-m}(B(c, r))$ (resp. $f^{-m}(B(c, r_0))$) containing~$c_0$.

\partn{1}
Since the degree of~$f^m$ on~$U$ and the degree of~$f^m$ on~$\hU$ are the same, there is a unique connected component of
$$ f^{-m}\left( B \left( c, \eta_1^{(1 + 5\widehat{\ell_{\max}(f)})\ell} r_0 \right) \right)
\mbox{ and of }
f^{-m} \left( B \left( c, \eta_1^{5\widehat{\ell_{\max}(f)}\ell} r_0 \right) \right)$$
contained in~$\hU$.
We will denote it by~$\hB$ and~$B$, respectively.
It follows that the degree of~$f^m$ on each of the sets~$U, \hU, B, \hB$ is the same.
Thus, by Fact~\ref{f:Jacobian} we have
$$ \frac{\rho_f(\hB)}{\rho_f(B)}
=
\frac{\rho_f(f^m(\hB))}{\rho_f(f^m(B))}. $$
Using part~2 of Lemma~\ref{l:inverse doubling} inductively we obtain,
\begin{multline}\label{e:concentration}
\rho_f (f^m(\hB))
=
\rho_f\left(B \left(c, \eta_1^{(1 + 5 \widehat{\ell_{\max}(f)}) \ell} r_0 \right)\right)
\\ \ge
2^\ell \rho_f \left( B \left( c, \eta_1^{5\widehat{\ell_{\max}(f)}\ell} r_0 \right)  \right)
=
2^\ell \rho_f(f^m(B))
>
C_*^k \rho_f (f^m(B)).
\end{multline}

\partn{2}
Let~$\varphi_0$ be a Moebius transformation such that~$\varphi_0(0) = c$ and
$$ \varphi_0(\{ z \in \C : |z| < 1 \})
=
B \left( c, \eta_1^{5\widehat{\ell_{\max}(f)}\ell} r_0 \right). $$
Using the inequalities $\eta_1^\ell \exp(1/10) \le 2^{(N + 1)/2} \exp(1/10) < 2^{N}$ we obtain using~\eqref{e:modulus defect}
\begin{multline*}
B \left( c, \eta_1^{(1 + 5\widehat{\ell_{\max}(f)})\ell} r_0 \right)
\subset
\varphi_0 \left( \left\{ z \in \C : |z| < \eta_1^\ell \exp(1/10) \right\} \right)
\\ \subset
\varphi_0(\{ z \in \C : |z| < 2^{N} \}).
\end{multline*}
On the other hand, using the inequalities $\eta_1^{\ell} \ge 2^{N/2}$ and $\eta_1^{\ell \widehat{\ell_{\max}}(f)} \ge \exp(1 /10)$ we obtain using~\eqref{e:modulus defect}
\begin{multline*}
B(c, r_0)
\subset
\varphi_0 \left( \left\{ z \in \C : |z| < \eta_1^{- 5\widehat{\ell_{\max}(f)}\ell} \exp(1/10) \right\} \right)
\\ \subset
\varphi_0 \left( \left\{ z \in \C : |z| < 2^{- N \widehat{\ell_{\max}}(f)} \right\} \right)
\end{multline*}
and
\begin{multline*}
B(c, r)
\supset
\varphi_0 \left( \left\{ z \in \C : |z| < \eta_1^{(1 + 5\widehat{\ell_{\max}(f)})\ell} \exp(-1/10) \right\} \right)
\\ \supset
\varphi_0 \left( \left\{ z \in \C : |z| < 4^{N \widehat{\ell_{\max}}(f)} \right\} \right).
\end{multline*}
We have shown that~$f^m(\hB \setminus B) \subset \varphi_0(A(1, 2^N))$ and that
$$ \varphi_0 \left( A \left(2^{- N \widehat{\ell_{\max}}(f)}, 4^{N \widehat{\ell_{\max}}(f)} \right) \right)
\subset
B(c, r) \setminus \overline{B(c, r_0)}
=
f^m(A). $$
Since the degree of~$f^m : A \to B(c, r) \setminus \overline{B(c, r_0)}$ is at least~$N$ and at most~$\widehat{\ell_{\max}}(f) N$, we conclude that there is a univalent map~$\varphi : A(1, 8) \to A$ such that $\hB \setminus \overline{B} \subset \varphi(A(2, 4))$.
So there is~$r' > 0$ such that
$$ B(c_0, r')
\subset
B
\subset
\hB
\subset
B(c_0, Mr')
\subset
B(c_0, 2^k r'). $$
Using~\eqref{e:concentration} we obtain
$$ \rho_f(B(c_0, r'))
\le
\rho_f(B)
<
C_*^{-k} \rho_f(\hB)
\le
C_*^{-k} \rho_f(B(c_0, 2^kr')). $$
This contradicts the doubling property of~$\rho_f$ on~$J(f)$ and completes the proof of the theorem.
\end{proof}

\section{John domains as Fatou components}
\label{s:John domains}
This section is devoted to the proof of Theorem~\ref{t:John domains}.
We also prove Fact~\ref{f:cubic Newton} at the end of this section.

We start by recalling the definition of John domain.
An open and connected subset~$D$ of the Riemann sphere~$\CC$ is a \emph{John domain}, if there is~$z_0 \in D$ and a constant~$C > 0$ such that the following property holds: for each~$z \in \partial D$ there is a path~$\gamma$ in~$\CC$ joining~$z_0$ and~$z$, such that~$\gamma \setminus \{ z \} \subset D$ and such that for each~$w \in \gamma$ we have
$$ \dist(w, \partial D) \ge C \dist(w, z). $$

We will say that a subset~$J$ of~$\CC$ is \emph{porous at a point~$x$} in~$J$, if there is~$\xi \in (0, 1)$ such that for each small~$r > 0$ there is~$y \in B(x, r)$ such that~$B(y, \xi r) \cap J = \emptyset$.
Note that if~$D$ is a John domain, then the set~$\partial D$ is porous at each of its points.
\begin{lemm}
  \label{l:non-porosity}
Let~$f$ be a rational map of degree at least two having arbitrarily small nice couples of arbitrarily large modulus.
Then for each recurrent critical point~$c_0$ in~$J(f)$, the Julia set~$J(f)$ is not porous at~$c_0$.
\end{lemm}
In the proof of this lemma we use Lemma~\ref{l:thin unramified annulus}.
This last result implies that for each~$N \ge 1$ there are integers $M \in \{ N, \ldots, N \widehat{\ell_{\max}}(f) \}$ and~$m \ge 1$ such that~$f^m$ is in local coordinates close to the map~$z \mapsto z^M$ on a small think annulus around~$c_0$.
The fact that~$J(f)$ is uniformly perfect allows us estimate the ``non-porosity'' of~$J(f)$ at~$c_0$ on this thick annulus.
Taking~$N$ arbitrarily large will allow us to conclude that~$J(f)$ is not porous at~$c_0$.

We will introduce some notation to prove Lemma~\ref{l:non-porosity}.
For~$\tau \in \C$ we denote by~$T_\tau : \C \to \C$ the translation $T_\tau(z) = z + \tau$ and for~$\lambda \in \C \setminus \{ 0 \}$ the homothecy $M_\lambda(z) = \lambda z$.
For~$s, s' \in \R$ with~$s < s'$ we put
$$ S(s, s') \= \{ z \in \C : s < \Im(z) < s' \}. $$

Recall that~$\CC$ is endowed with the spherical metric.
We identify~$\CC$ with~$\C \cup \{ \infty \}$.
For a point~$x \in \CC$ we denote by~$\widehat{x} \in \CC$ the antipodal point of~$x$ and we let~$\psi_{x}$ be an isometry of~$\CC$ mapping~$0$ to~$x$.
Furthermore we put
\begin{center}
\begin{tabular}{rcl}
$E_x : \C$ & $\to$ & $\CC \setminus \{ x, \widehat{x} \} $ \\
$z$ & $\mapsto$ & $\psi_x (\exp(- 2\pi i z))$.
\end{tabular}
\end{center}
It is a $1$-periodic covering map.
\begin{proof}[Proof of Lemma~\ref{l:non-porosity}]
We will show that for each~$s_0 \in \R, h_0 > 0$ and $\varepsilon_0 > 0$, there is $s < s_0$ such that~$E_{c_0}^{-1}(J(f))$ is $\varepsilon_0$-dense in the strip~$S(s - h_0, s)$.
This clearly implies that~$J(f)$ is not porous at~$c_0$.

Let~$h_0' > 0$ be a sufficiently large constant such that for every univalent map~$\tvarphi : S(0, h_0') \to \C$ that commutes with the translation~$T_1$, there is $s \in \R$ such that the image of~$\tvarphi$ contains the strip~$S(s - h_0, s)$, see~\cite[Proposition~$2.1$]{McM94}.
Taking~$h_0' > 0$ larger if necessary we assume~$h_0' \ge \varepsilon_0$.
It follows from Koebe distortion theorem that there is a constant~$D > 1$ such that for each~$h \ge \varepsilon_0$ and each univalent map $\tvarphi : S(-2h, 3h) \to  \C$ that commutes with~$T_1$ the distortion of~$\tvarphi$ on $S(-h, 2h)$ is bounded by~$D$.

Since the Julia set~$J(f)$ is uniformly perfect~\cite{Hin93,MandRo92}, it follows that there are constants~$s_1 \in \R$ and $h_1 > 0$ such that for each~$x \in J(f)$ and~$s \le s_1$, the set $E_x^{-1}(J(f))$ intersects the strip~$S(s, s + h_1)$.
Since the set~$E_x^{-1}(J(f))$ is $1$-periodic, it follows that the set~$E_x^{-1}(J(f))$ is $(h_1 + 1)$\nobreakdash-dense in $\{ z \in \C : \Im(z) \le s_1 \}$.
Decreasing~$s_0$ if necessary we assume~$s_0 \le s_1$.

Let~$N \ge \varepsilon_0^{-1} D (h_1 + 1)$ be an integer, put~$h_2 \= 5 N \widehat{\ell_{\max}}(f) h_0'$ and let~$r_* > 0$ and~$\kappa \in (0, 1)$ be sufficiently small so that for each~$r \in (0, r_*)$ and~$x \in J(f)$ the set~$E_x^{-1}(B(x, r))$ is contained in $\{ z \in \C : \Im(z) \le s_0 \}$ and there is~$s \le s_0$ such that the set~$E_x^{-1}(B(x, r) \setminus B(x, \kappa r))$ contains the strip~$S(s - h_2, s)$.

Let $m, c, r$ and~$A$ be given by Lemma~\ref{l:thin unramified annulus} for the choices of~$N$, $r_*$ and~$\kappa$ as above, and let~$d$ be the degree of~$f^m : A \to B(c, r) \setminus \overline{B(c, \kappa r)}$.
We have~$N \le d \le \widehat{\ell_{\max}}(f) N$.
Let~$s \le s_0$ be such that~$E_c^{-1}(B(c, r)  \setminus \overline{B(c, \kappa r)})$ contains the strip~$S(s - h_2, s)$.
If we denote by~$A'$ the pull-back of the annulus~$A_0 \= E_c(S(s - h_2, s))$ by~$f^m$ contained in~$A$ and put~$\widetilde{A'} \= E_{c_0}^{-1}(A')$, then
$$ f^m \circ E_{c_0} : \widetilde{A'} \to A_0
\text{ and }
E_c : S(s - h_2, s) \to A_0, $$
are both universal covering maps.
Thus there is a biholomorphic map
$$ \tvarphi : S \left( s - h_2, s \right) \to \widetilde{A'}, $$
satisfying~$f^m \circ E_{c_0} \circ \tvarphi = E_c$ and therefore~$\tvarphi \circ T_d = T_1 \circ \tvarphi$.
In particular, the map $\hvarphi : S \left( \frac{s - h_2}{d}, \frac{s}{d} \right) \to \widetilde{A'}$ defined by~$\hvarphi \= \tvarphi \circ M_{d}$ commutes with~$T_1$.

Since~$d \le N \widehat{\ell_{\max}}(f)$, we have~$h_2/d \ge 5h_0'$.
So, if we put~$s' = s/d - 3h_0'$, then~$S(s' - 2h_0', s + 3h_0') \subset S \left( \tfrac{s - h_2}{d}, \tfrac{s}{d} \right)$ and therefore the distortion of~$\hvarphi$ on the strip~$S(s' - h_0', s' + 2h_0')$ is bounded by~$D$.
By our choice of~$h_0'$ there is~$s'' \in \R$ such that the set~$\hvarphi(S(s', s' + h_0'))$ contains the strip~$S(s'' - h_0, s'')$.
Since
$$ E_{c_0}(\hvarphi(S(s' - h_0', s' + 2h_0'))) \subset A \subset B(c_0, r_*), $$
by our choice of~$r_*$ we have~$s'' \le s_0$.
On the other hand, since the set~$E_c^{-1}(J(f))$ is~$(h_1 + 1)$-dense in~$S(s - h_2, s)$ and~$(h_1 + 1)/d \le \varepsilon_0 D^{-1} \le \varepsilon_0 \le h_0'$, it follows that the set
$$ M_{d}^{-1}(E_c^{-1}(J(f)) \cap S(s - h_2, s)), $$
is $(\varepsilon_0 D^{-1})$-dense in the strip~$S(s', s' + h_0')$.
Thus the set
$$\hvarphi \circ M_d^{-1} \left( E_c^{-1}(J(f)) \cap S(s - h_2, s)) \right)
\subset
E_{c_0}^{-1}(J(f)), $$
is $\varepsilon_0$-dense in~$\hvarphi(S(s', s' + h_0'))$ and hence in~$S(s'' - h_0, s'')$.
This completes the proof of the lemma.
\end{proof}

\begin{proof}[Proof of Theorem~\ref{t:John domains}]
If~$f$ semi-hyperbolic, then by~\cite[Theorem~$1$]{Mih0803} the attracting basin~$\cA$ is a John domain, see also~\cite[Theorem~$1.2$]{Yin99} for the case of connected Julia sets.
Suppose now that~$\cA$ is a John domain.
By~\cite[Theorem~$1.1$]{Yin99}, to show that~$f$ is semi-hyperbolic it is enough to show that~$f$ has neither parabolic periodic points nor recurrent critical points in its Julia set.
Since~$\cA$ is a John domain, the set~$J(f) = \partial \cA$ is porous at each of its points.
On the other hand, $\cA$ is a H{\"o}lder domain, see for example~\cite[\S$5.2$]{Pom92}.
So, combining~\cite[Theorem~$1$, (iii)]{GraSmi98} and~\cite[Main Theorem]{PrzRivSmi03}, we obtain that~$f$ satisfies the TCE condition.
Therefore~$f$ does not have parabolic periodic points and by~\cite[Proposition~$4.2$]{PrzRiv07} it satisfies the hypothesis of Lemma~\ref{l:non-porosity}.
Since~$J(f)$ is porous at each of its points, Lemma~\ref{l:non-porosity} implies that~$f$ does not have recurrent critical points in~$J(f)$.
This shows that~$f$ is semi-hyperbolic and completes the proof of the theorem.
\end{proof}

\begin{proof}[Proof of Fact~\ref{f:cubic Newton}]
We will use that the Julia set of~$N$ is connected, see for example~\cite[Proposition~$2.6$]{Lei97}.

The rational map~$N$ has four critical points, three of which are fixed.
Denote the corresponding immediate attracting basins by~$B_0$, $B_1$ and~$B_2$.
If there is another attracting cycle, then~$N$ is hyperbolic and the result is well known.
We thus assume that there is no other attracting cycle besides the three fixed critical points.
Since by hypothesis~$N$ has no parabolic periodic point or Siegel disks it follows that every point in the Fatou set is eventually mapped into~$B_0 \cup B_1 \cup B_2$ under forward iteration.
So we just need to show that for each~$i \in \{ 0, 1, 2 \}$ the Fatou components of~$N$ that are eventually mapped to~$B_i$ are quasi-disks with a uniform constant.

Let~$P$ the closure of the forward orbit of the fourth critical point.
Since~$N$ is twice renormalizable, for each~$i \in \{ 0, 1, 2 \}$ the boundary of~$B_i$ is disjoint from~$P$, see the proof of~\cite[Proposition~$8.3$, 4)]{Roe08}.
Since the Julia set of~$N$ is connected, it follows that~$B_i$ is a quasi-disk.
Fix a simply-connected neighborhood~$V_i$ of~$\overline{B_i}$ that is disjoint from~$P$.
Let~$U$ be a Fatou component of~$N$ different from~$B_i$ that is eventually mapped to~$B_i$ and denote by~$n \ge 1$ the least integer such that~$f^n(U) = B_i$.
Then~$f^n$ maps a neighborhood of~$U$ univalently onto~$V_i$ and Koebe distortion theorem implies that~$U$ is a quasi-disk whose constant depends on~$B_i$ and~$U_i$ only.
\end{proof}

\bibliographystyle{alpha}
\bibliography{$HOME/papers/0BIB/papers}
\end{document}